\documentclass[12pt,a4paper]{amsart}
\usepackage{amsmath,amssymb,amsthm,booktabs,comment,longtable}
\usepackage{enumitem}
\usepackage{lscape}
\usepackage{graphicx}
\usepackage[all]{xy}
\usepackage{ccaption}
\usepackage{stmaryrd}
\theoremstyle{plain}
\newtheorem{theorem}{Theorem}[section]

\newtheorem{proposition}[theorem]{Proposition}
\newtheorem{definition}[theorem]{Definition}
\newtheorem{fact}[theorem]{Fact}

\newtheorem{remark}[theorem]{Remark}

\newtheorem*{claim}{Claim}
\newtheorem{notation}[theorem]{Notation}

\newcommand{\Ad}{\mathop{\mathrm{Ad}}\nolimits}

\newcommand{\ev}{\mathop{\mathrm{ev}}\nolimits}

\newcommand{\Tr}{\mathop{\mathrm{Tr}}\nolimits}

\newcommand{\R}{\mathbb{R}}
\newcommand{\Z}{\mathbb{Z}}

\newcommand{\C}{\mathbb{C}}

\newcommand{\pairing}[2]{\langle #1, #2 \rangle}
\newcommand{\Sym}{\operatorname{Sym}}
\newcommand{\trans}{{}^t\!}
\newcommand{\TODO}{{\bf TODO}}
\newcommand{\spanvec}{\operatorname{span}}
\newcounter{mycounter}

\begin{document}

\title[A method to construct exponential families]
{A method to construct exponential families by representation theory}
\author{Koichi Tojo and Taro Yoshino}
\address{RIKEN Center for Advanced Intelligence Project, Nihonbashi 1-chome Mitsui Building, 15th floor, 1-4-1 Nihonbashi, Chuo-ku, Tokyo 103-0027, Japan}
\email{koichi.tojo@riken.jp}
\address{Graduate School of Mathematical Science, The University of Tokyo,\\ 3-8-1 Komaba, Meguro-ku, Tokyo 153-8914, Japan
}
\email{yoshino@ms.u-tokyo.ac.jp}

\date{\today}
\maketitle

\begin{abstract}
In this paper, we give a method to construct ``good'' exponential families systematically by representation theory. 
More precisely, we consider a homogeneous space $G/H$ as a sample space and 
construct an exponential family invariant under the transformation group $G$ by using a representation of $G$. 
The method generates widely used exponential families such as 
normal, gamma, Bernoulli, categorical, Wishart, von Mises, Fisher--Bingham and hyperboloid distributions. 

\end{abstract}

\setcounter{tocdepth}{3}

\tableofcontents

\section{Introduction}\label{sec:intro}
Exponential families are important subjects of study in the field of 
information geometry and are often used for Bayesian inference 
because they have conjugate priors, which have been well studied (\cite{diaconis-ylvisoker}). 
By definition, there are infinitely many exponential families, 
however, only a small part of them are widely used such as 
Bernoulli, categorical, normal, multivariate normal, gamma, inverse gamma, von Mises, von Mises--Fisher, Fisher--Bingham, Wishart and hyperboloid distributions, which have been investigated individually. 

We want to understand them systematically. 
It should be nice if we have a method to construct exponential families which has the following two properties: 
\begin{enumerate}
\item The method can generate many well-known good distributions. 
\item Distributions obtained by the method are limited in principle. 
\end{enumerate}
In this paper, we introduce a method satisfying them. 
In our method, representation theory plays an important role. 
Actually, most of useful distributions have the same symmetry as the sample spaces. 
More precisely, the sample space can be regarded as a homogeneous space $G/H$ and the family of distributions is $G$-invariant. 

In Section~\ref{sec:general_method}, 
we introduce our method and show that the obtained families are invariant under the transformation grourps $G$. 
In Section~\ref{sec:examples}, 
by the method, we realize examples as mentioned above. 

\section{Related work}
\subsection{Transformation model}
A transformation model is one of the parametric models in statistics, 
which enables us to understand many families of probability measures on sample spaces comprehensively and 
have good properties on statistics (\cite{Barndorff-Nielsen89}). 
However, transformation model has a problem that  
the overlap between transformation models and exponential families is too small. 
Therefore, it fails to include many important families of probability measures on sample spaces which admit group actions (see Section~1.3 in \cite{Barndorff-Nielsen78a}). 
Our method clear up the problem by using representation theory. 
In fact, our method generates many useful exponential families invariant under a transformation group (see Table~\ref{tab:1}). 


\subsection{Harmonic exponential families}
Our method can be regarded as a generalization of harmonic exponential families suggested in \cite{cohen-welling}. 
They studied densities 
on a compact Lie group $G$ or its homogeneous spaces $G/H$ such as a torus or a sphere
by using unitary representations of $G$. 
As mentioned in Section~7 (``Discussion and Future Work'') in \cite{cohen-welling}, 
they wished an extension to noncompact groups. 
Our method results in a generalization of their method including noncompact groups. 
In the case where $G$ is compact, there is a unique invariant measure on $G$ or $G/H$ up to positive constant. 
On the other hand, in the case where $G$ is noncompact, it does not hold. 
For example, let $G$ be the affine transformation group $\R^\times \ltimes \R$ and $H$ the closed subgroup $\R^\times$, then $G/H\simeq \R$ admits no $G$-invariant measures. 
However, $G/H$ admits relatively $G$-invariant measures. 
Therefore, we use a relatively $G$-invariant measures on $G/H$ instead of a $G$-invariant measure. 
Since a relatively $G$-invariant measure on $G/H$ is not unique, we have a lot of choice of relatively $G$-invariant measures. 
Our method does not choose one of them but take ``all'' of them (see Subsection~\ref{subsec:our_method}). 

\section{A method to construct good exponential families}\label{sec:general_method}
In this section, we introduce a method to construct a family of distributions on a homogeneous space $G/H$ by representation theory. 
Here $G$ is a Lie group with finitely many connected components and $H$ is a closed subgroup of $G$. 
Moreover, we show that the families obtained by our method are exponential families on $G/H$. 
\subsection{Our method}\label{subsec:our_method}
We put $X:=G/H$ as above, which is regarded as a sample space. 
Our method generates a family $\{p_\theta\}_{\theta\in \Theta}$ of probability measures on $X$ from
two inputs as follows:
\begin{enumerate}
\item Let $\rho\colon G\to GL(V)$ be a  representation of $G$ on a finite dimensional real vector space $V$. 
\item Let $v_0$ be an $H$-fixed vector in $V$. 
\end{enumerate}
We consider the following vector space (see Fact~\ref{fact:fin_dim}): 
\[W_0(G,H):=\{ \tau\colon  G\to \R\ |\ \tau \text{ is a continuous group homomorphism with } \tau|_H=0\}. \]
We assume that there exist nonzero relatively $G$-invariant measures on $G/H$ and take one of them, say $\nu$. 
The family $\mathcal{P}$ below does not depend on the choice of $\nu$
(Remark~\ref{remark:choice_of_nu}). 
Then we set 
\begin{align*} 
d\tilde{p}_\theta(x)&:=d\tilde{p}_{\varphi,\tau}(x)=\exp(-\pairing{\varphi}{\rho(x) v_0}+\tau(x))d\nu(x)\\
&\hspace{1cm} (\theta=(\varphi, \tau)\in V^\vee\times W_0(G,H), x=gH\in G/H), \\
\Theta&:=\left\{ \theta=(\varphi,\tau)\in V^\vee \times W_0(G,H)\ \Big|\ \int_X d\tilde{p}_{\theta}<\infty \right\}, \\
c_\theta&:= \left(\int_X d\tilde{p}_\theta\right)^{-1}\quad (\theta \in \Theta). 
\end{align*}
Here we denote the dual space of $V$ by $V^\vee$. 
From the conditions that $v_0$ is $H$-fixed and $\tau|_H=0$, 
the notations $\rho(x) v_0$ and $\tau(x)$ are well-defined. 
Then we get a probability measure $p_\theta$ on $X$ as follows:
\[ p_\theta :=c_\theta \tilde{p}_\theta\ (\theta \in \Theta).  \]
As a result, we obtain the following family of distributions on $X$:
\[ \mathcal{P}:=\{ p_\theta \}_{\theta \in \Theta}. \]

\begin{remark}\label{remark:choice_of_nu}
This family $\mathcal{P}$ does not depend on the choice of $\nu$. 
In fact, for another nonzero relatively $G$-invariant measure $\nu'$ on $G/H$, 
there exist $c\in \R_{>0}$ and $\tau\in W_0(G,H)$ such that $d\nu'(x)=c e^{\tau(x)}d\nu(x)$ by the next fact. 
\end{remark}

\begin{fact}[See {\cite[Corollary~4.27]{ty21}} for example]\label{fact:relatively_G_invarinat_measures}
Assume that there exist nonzero relatively $G$-invariant measures on $G/H$, and take one of them, say $\nu_0$. 
Then, 
the set of all nonzero relatively $G$-invariant measures on $G/H$ is given as follows:
\begin{align*}
\{c e^\tau \nu_0 \ |\ c\in \R_{>0}, \tau \in W_0(G,H) \}. 
\end{align*}
\end{fact}

\subsection{Exponential family}\label{subsec:pf_expf}
In this subsection, we see that families obtained by our method are exponential families. 
There are several definitions of an exponential family. 
We adopt the following Definition~\ref{def:expf} and show that families of distributions obtained by our method are exponential families in this sense. 

Let $X$ be a manifold 
and $\mathcal{R}(X)$ the set of all Radon measures on $X$. 

\begin{definition}[exponential family, see Section~5 in \cite{Barndorff-Nielsen70} for example]\label{def:expf}
Let $\mathcal{P}$ be a non-empty subset of $\mathcal{R}(X)$ consisting of probability measures. 
$\mathcal{P}$ is called an \emph{exponential family} on $X$ if 
there exists a triple $(\mu, V, T)$ satisfying the following four conditions:
\begin{enumerate}
\item $\mu\in \mathcal{R}(X)$,
\item $V$ is a finite dimensional real vector space, 
\item $T\colon X\to V$ is a continuous map, 
\item for any $p\in \mathcal{P}$, there exists $\theta\in V^\vee$ such that 
\[ dp(x)=c_\theta e^{-\pairing{\theta}{T(x)}}d\mu(x), \]
where $V^\vee$ is the dual space of $V$, and 
\[ c_\theta:=\left(\int_{x\in X} e^{-\pairing{\theta}{T(x)}}d\mu(x)\right)^{-1}. \]
\end{enumerate}
\end{definition}

\begin{proposition}\label{prop:ours_are_expf}
A family $\mathcal{P}:=\{ p_\theta\}_{\theta\in \Theta}$ obtained by the method in Section~\ref{subsec:our_method} is an exponential family on $X$ if $\Theta$ is not empty. 
\end{proposition}
To show this proposition, we need: 
\begin{fact}[{\cite[Proposition~3.3]{ty21}}]
\label{fact:fin_dim}
Let $G$ be a Lie group with finitely many connected components. 
Then the following set is a finite dimensional vector subspace of $C(G)$:
\[ \{ f\in C(G)\ |\ f\colon G\to \R \text{ is a group homomorphism} \}. \]
Here, $C(G)$ denotes the set of all $\R$-valued continuous functions on $G$. 
\end{fact}

\begin{proof}[Proof of Proposition~\ref{prop:ours_are_expf}]
From Fact~\ref{fact:fin_dim}, 
$W_0(G,H)$ is a finite dimensional real vector space. 
Therefore $\tilde{V}:=V\oplus W_0(G,H)^\vee$ is also a finite dimensional vector space. 
Then we define $T\colon X\to \tilde{V}$ by 
\begin{align*}
T(x):=(\rho(x)v_0, \ev_x), 
\end{align*}
where $\ev_x \colon W_0(G,H)\to \R$ is the evaluation map at $x\in X$. 
We regard $\tilde{V}$ as $V$ of the condition (ii) in Definition~\ref{def:expf}. 
Then, the other conditions in Definition~\ref{def:expf} are satisfied. 
\end{proof}

\subsection{$G$-invariance}
In this subsection, we see that families obtained by our method are $G$-invariant. 
We put $X:=G/H$, then we have a natural $G$-action on $\mathcal{R}(X)$ induced by pushforward. 
\begin{definition}
A family $\mathcal{P}\subset \mathcal{R}(X)$ is said to be $G$-invariant if 
\[ g\cdot p \in \mathcal{P} \text{ for any }p\in \mathcal{P} \text{ and }g\in G. \]
\end{definition}

\begin{proposition}
The family $\mathcal{P}:=\{ p_\theta\}_{\theta\in \Theta}$ obtained by the method in Section~\ref{subsec:our_method} is $G$-invariant. 
Moreover, the $G$-action on $\Theta$ is given by 
\begin{align*}
G\times \Theta\to \Theta, (g, (\varphi, \tau))\mapsto (\rho^\vee(g)\varphi, \tau). 
\end{align*}
Here, $\rho^\vee$ is the contragredient representation of $\rho$. 
\end{proposition}

\begin{proof}
Take any $\theta:=(\varphi, \tau)\in \Theta$ and $g\in G$. 
Put $\theta':=(\rho^\vee(g)\varphi, \tau)\in V^\vee \times W_0(G,H)$. 
We show that $\theta'\in \Theta$ and $g\cdot p_\theta=p_{\theta'}$. 
We have 
\begin{align*}
d(g\cdot p_\theta)(x)&=c_\theta\exp(-\pairing{ \varphi }{\rho (g^{-1}x)v_0}+\tau (g^{-1}x))d(g\cdot \nu)(x)\\
&=c_\theta \exp(-\pairing{ \varphi}{\rho(g^{-1})\rho(x) v_0}+\tau(g^{-1})+\tau(x)) \chi(g)^{-1}d\nu(x)\\
&=c_\theta \exp(-\tau(g))\chi(g)^{-1}\exp(-\pairing{\rho^\vee(g)\varphi}{\rho(x)v_0}+\tau(x))d\nu(x)\\
&=c_\theta \exp(-\tau(g))\chi(g)^{-1} d\tilde{p}_{\theta'}(x). 
\end{align*}
Here, $\chi\colon G\to \R_{>0}$ is the multiplier of $\nu$, namely, 
it satisfies $d(g\cdot \nu)(x)=\chi(g)^{-1}d\nu(x)$ for any $g\in G$. 
From the equality above, $\theta'\in \Theta$ holds. 
One should check that
\[c_\theta\exp(-\tau(g))\chi(g)^{-1}=c_{\theta'}, \]
but in fact it is not necessary. It is enough to observe that $p_\theta$ and $p_{\theta'}$ are probability measures.
Thus, $\mathcal{P}$ is $G$-invariant. 
\end{proof}

\section{Examples}\label{sec:examples}
In this section, we realize widely used exponential families in Table~\ref{tab:1} by our method. 

\begin{table}[h]
\caption{Examples and inputs $(G,H,V,v_0)$ for them}
\small
\begin{tabular}{c|c|c|c|c|c}\label{tab:1}
distribution & sample space & $G$ & $H$ & $V$ & $v_0$ \\
\hline
Bernoulli & $\{\pm 1\}$ & $\{\pm 1\}$ & $\{1\}$& $\R_{sgn}$ & 1 \\
Categorical &$\{1,\cdots,n\}$ &$\mathfrak{S}_n$ & $\mathfrak{S}_{n-1}$ & $W$ & $w$\\
Normal & $\R$ & $\R^\times \ltimes \R$ & $\R^\times $ & $\Sym(2,\R)$ & $E_{22}$ \\
Multivariate normal & $\R^n$ &$GL(n,\R)\ltimes \R^n$ & $GL(n,\R)$ & $\Sym(n+1,\R)$ & $E_{n+1,n+1}$\\
Gamma  & $\R_{>0}$ & $\R_{>0}$ & $\{1\}$ & $\R_1$ & 1\\
Inverse gamma & $\R_{>0}$ & $\R_{>0}$ &$\{1\}$ & $\R_{-1}$ & 1\\
Wishart & $\Sym^+(n,\R)$ & $GL(n,\R)$ & $O(n)$ & $\Sym(n,\R)$ & $I_n$ \\
Von Mises & $S^1$ & $SO(2)$ & $\{I_2\}$ & $\R^2$ & $e_1$\\
Von Mises--Fisher & $S^{n-1}$ & $SO(n)$ & $SO(n-1)$ & $\R^n$ & $e_1$\\
Fisher--Bingham & $S^{n-1}$ & $SO(n)$ & $SO(n-1)$ & $\R^n\oplus \Sym(n,\R)$ & $e_1\oplus E_{11}$ \\
Hyperboloid & $H^n$ & $SO_0(1,n)$ & $SO(n)$ & $\R^{n+1}$ & $e_0$\\
Poincar\'e  & $\mathcal{H}$ & $SL(2,\R)$ & $SO(2)$ & $\Sym(2,\R)$ & $I_2$
\vspace{2mm}
\end{tabular}
Here $W:=\{(x_1,\cdots, x_n)\in \R^n\ |\ \sum_{i=1}^n x_i=0\}$, 
$w:=(-(n-1),1,\cdots, 1)\in W$. 
\end{table}

We use Fact~\ref{prop:triv_omega} and Remark~\ref{rem:factor} below to determine $W_0(G,H)$ for pairs $(G,H)$ in the following subsections. 

\begin{fact}[{\cite[Proposition~4.28]{ty21}}]\label{prop:triv_omega}
Let $G$ be a Lie group with finitely many connected components and $H$ a closed subgroup of $G$. 
We denote by $\frak{g}$ and $\frak{h}$ the Lie algebras of $G$ and $H$, respectively. 
Then we have $W_0(G,H)=\{0\}$ if the pair $(G,H)$ satisfies one of the following conditions:
\begin{enumerate}
\item[(i)] $\frak{g}=\frak{h}+[\frak{g},\frak{g}]$, 
\item[(ii)] $G$ is compact, 
\item[(iii)] $G$ is semisimple. 
\end{enumerate}
\end{fact}

\begin{remark}\label{rem:factor}
Let $G$ be a group and 
$\tau\colon G\to \R$ a group homomorphism. 
Then $\tau$ factors $G/[G,G]$. 
\end{remark}

\subsection{Normal and multivariate normal distributions}
In this subsection, we realize the families of normal distributions on $\R$ and multivariate normal distributions on $\R^n$ by using a representation of an affine transformation group. 
\begin{proposition}
Put $G=\R^\times \ltimes \R$ and $H=\R^\times \subset G$. 
Let $\rho\colon G\to GL(\Sym(2,\R))$ be a representation given as follows:
\[ \rho(g)(S):=\begin{pmatrix}a & b\\ 0 & 1 \end{pmatrix}S{\vphantom{\begin{pmatrix}a & b\\ 0 &1\end{pmatrix}}}^{\ t}\!\!\begin{pmatrix}a & b\\ 0 & 1 \end{pmatrix}\quad (g=(a, b)\in \R^\times\ltimes \R, S\in\Sym(2,\R)). \]
Then $v_0:=\begin{pmatrix}0 & 0\\ 0 & 1\end{pmatrix}\in \Sym(2,\R)$ is $H$-fixed, 
and by our method we obtain the family of normal distributions on $\R$, 
\begin{align} \left\{\tfrac{1}{\sqrt{2\pi\sigma^2}}\exp\left(-\tfrac{(x-\mu)^2}{2\sigma^2}\right)dx  \right\}_{(\sigma, \mu)\in \R_{>0}\times \R}. \label{fami:normal_dist}\end{align}
\end{proposition}

\begin{proof}
We identify $G/H=(\R^\times \ltimes \R)/\R^\times$ with $\R$ by the following homeomorphism:
\[ G/H\to \R,\ gH\mapsto b\quad \left(g=(a,b)\in G \right). \]

We have $W_0(G,H)=\{ 0\}$ from Fact~\ref{prop:triv_omega} (i). 
In fact, $\frak{g}=\R\oplus \R$, $\frak{h}=\R\oplus \{0\}$ and $[\frak{g},\frak{g}]=\{0\}\oplus \R$. 
We identify $\Sym(2,\R)^\vee$ with $\Sym(2,\R)$ by taking the following inner product on $\Sym(2,\R)$: 
\[ \pairing{x}{y}:=\Tr(xy) \quad (x,y\in \Sym(2,\R)). \]
Let $dx$ denote the Lebesgue measure on $\R$, which is relatively $G$-invariant. 
Then we have 
\begin{align*} 
d\tilde{p}_{\theta_1, \theta_2, \theta_3}(x)
&= \exp\left(-\Tr\left(\begin{pmatrix} \theta_1 & \theta_2 \\ \theta_2 & \theta_3\end{pmatrix}\rho(g)v_0\right)\right)dx\\
&=\exp(-(\theta_1 x^2+2\theta_2 x +\theta_3))dx\quad (g=(a, x)\in \R^\times \ltimes \R ). 
\end{align*}
The following equalities are well-known:
\begin{align*}
\Theta&=\left\{\begin{pmatrix}\theta_1&\theta_2\\ \theta_2&\theta_3\end{pmatrix}\in \Sym(2,\R)\ \Big| \int_\R d\tilde{p}_{\theta_1,\theta_2,\theta_3}<\infty \right\}\\
&=\left\{\begin{pmatrix}\theta_1&\theta_2\\ \theta_2&\theta_3\end{pmatrix}\in \Sym(2,\R)\ \Big|\ \theta_1>0\right\},\\
c_\theta&=\sqrt{\frac{\theta_1}{\pi}}\exp\left(\frac{\theta_1\theta_3-\theta_2^2}{\theta_1}\right). 
\end{align*}
Then we obtain the family of normal distributions (\ref{fami:normal_dist}) by a change of variables $\mu=-\frac{\theta_2}{\theta_1}$ and $\sigma=\frac{1}{\sqrt{2\theta_1}}$. 
\end{proof}

By generalizing the case above, we obtain the following:
\begin{proposition}
Let $n$ be a positive integer. 
Put $G=GL(n,\R)\ltimes \R^n$, $H=GL(n,\R)\subset G$. 
Let $\rho\colon G\to GL(\Sym(n+1,\R))$ be a representation given as follows:
\[ \rho(g)(S):=\begin{pmatrix}A & b\\ 0 &1 \end{pmatrix} S {\vphantom{\begin{pmatrix}a & b\\ 0 &1\end{pmatrix}}}^{\ t}\!\!\begin{pmatrix}A & b\\ 0 &1 \end{pmatrix}\quad (g=(A, b)\in GL(n,\R)\ltimes\R^n, S\in\Sym(n+1,\R)). \]
Then $v_0:=E_{n+1, n+1}\in \Sym(n+1,\R)$ is $H$-fixed, and 
by our method we obtain the family of multivariate normal distributions on $\R^n$, 
\begin{align}
\left\{ \frac{1}{\sqrt{(2\pi)^n \det\Sigma}}\exp\left(-\tfrac{1}{2}\trans(x-\mu)\Sigma^{-1}(x-\mu)\right)dx\right\}_{(\Sigma, \mu)\in \Sym^+(n,\R)\times \R^n} \label{fami:mult_normal}, 
\end{align}
where we denotes the set of all positive definite symmetric matrices of size $n$ by $\Sym^+(n,\R)$. 
\end{proposition}
\begin{proof}
We identify $G/H=(GL(n,\R)\ltimes \R^n) /GL(n,\R)$ with $\R^n$ by the following map:
\[ G/H\to \R^n, (A,b)H\mapsto b. \]
We have $W_0(G,H)=\{0\}$ 
from Fact~\ref{prop:triv_omega} (i). 
In fact, $\frak{g}=\R^n\oplus \R^n$, $\frak{h}=\R^n\oplus \{0\}$ and $[\frak{g}, \frak{g}]=\{0\}\oplus \R^n$. 
We identify $\Sym(n+1,\R)^\vee$ with $\Sym(n+1,\R)$ by taking the following inner product on $\Sym(n+1,\R)$: 
\[ \pairing{x}{y}=\Tr(xy) \quad (x,y\in \Sym(n+1,\R)). \]
Let $dx$ denote the Lebesgue measure on $\R^n$, which is relatively $G$-invariant. 
Then we have 
\begin{align*}
d\tilde{p}_{\theta_1,\theta_2,\theta_3}(x)
&=\exp\left(-\Tr\left(\begin{pmatrix}\theta_1 & \theta_2 \\ \trans \theta_2 & \theta_3 \end{pmatrix} \begin{pmatrix}x\trans x& x\\ \trans x & 1 \end{pmatrix}\right)\right) dx\\
&=\exp(-(\trans x \theta_1 x +2 \trans \theta_2 x+\theta_3))dx\ (\theta_1\in \Sym(n,\R), \theta_2\in \R^n, \theta_3\in \R). 
\end{align*}
The following equalities are well-known:
\begin{align*}
\Theta&=\left\{ (\theta_1,\theta_2,\theta_3)\in \Sym(n,\R)\times \R^n\times \R\ |\ \int_{\R^n}d\tilde{p}_{\theta_1,\theta_2,\theta_3}<\infty\right\}\\
&=\left\{ (\theta_1,\theta_2,\theta_3)\in \Sym(n,\R)\times \R^n\times \R\ |\ \theta_1 \text{ is positive definite. }\right\},\\
c_\theta&=\sqrt{\frac{\det \theta_1}{\pi^n}}\exp\left(\frac{\det\begin{pmatrix}\theta_1 & \theta_2\\ \trans \theta_2 & \theta_3\end{pmatrix}}{\det \theta_1}\right). 
\end{align*}
Then we obtain the family of multivariate normal distributions (\ref{fami:mult_normal}) by a change of variables $\mu=-\theta_1^{-1} \theta_2\in \R^n$ and $\Sigma=\frac{1}{2}\theta_1^{-1}\in \Sym^+(n,\R)$. 
\end{proof}

\subsection{Bernoulli and categorical distributions}
In this subsection, we realize the families of Bernoulli distributions on $\{\pm 1\}$ and categorical distributions on $\{1,\cdots, n\}$ 
by using the signature representation of $\{\pm 1\}$ and a natural representation of the symmetric group, respectively. 
\begin{proposition}
Put $G=\{\pm 1\}$ and $H=\{1\}$. 
Let $\rho\colon \{\pm 1\}\to \R^\times$ be the signature representation. 
Then $v_0:=1$ is $H$-fixed, and by our method, we obtain the family of Bernoulli distributions on $\{ \pm 1\}$, 
\begin{align}
\left\{ s^{\frac{1+x}{2}}(1-s)^{\frac{1-x}{2}}dc(x) \right\}_{s\in (0,1)}\quad (x\in \{ -1,1\}),  \label{fami:Bernoulli}
\end{align}
where $dc$ denotes the counting measure on $\{\pm 1\}$. 
\end{proposition}
\begin{proof}
Since $G$ is compact, we have $W_0(G,H)=\{0\}$ from Fact~\ref{prop:triv_omega} (ii). 
Therefore we have 
\[d\tilde{p}_\theta(x)=e^{-\theta x}dc(x)=\begin{cases}e^{-\theta}\quad (x=1) \\ e^\theta\quad (x=-1)\end{cases},\] where $dc$ is the counting measure on $G/H=\{\pm 1\}$, which is (relatively) $G$-invariant. 
This is integrable for any $\theta\in \R$, so we have 
$\Theta=\R$ and $c_\theta=\frac{1}{e^{-\theta}+e^\theta}$. 
Therefore we get 
\[ dp_\theta(x)=\begin{cases}\frac{e^{-\theta}}{e^{-\theta}+e^\theta}dc(x)\quad (x=1) \\ \frac{e^{\theta}}{e^{-\theta}+e^\theta}dc(x)\quad (x=-1)\end{cases}. \]
Thus by a change of variable $s=\frac{e^{-\theta}}{e^{-\theta}+e^\theta}$, 
we obtain the family of Bernoulli distributions (\ref{fami:Bernoulli}) on $\{\pm 1\}$. 
\end{proof}

By generalizing the case above, we obtain the following:
\begin{proposition}
Let $n$ be a positive number with $n\geq 2$, $G$ the symmetric group $\frak{S}_n$ and $H$ the stabilizer $\frak{S}_{n-1}=(\frak{S}_n)_1$ of $1$ in $G$. 
Let $\rho\colon G\to GL(W)$ be a subrepresentation of the natural representation $G\to GL(n,\R)$, where $W=\{ (x_1,\cdots, x_n)\in \R^n\ |\ \sum_{i=1}^n x_i=0\}$. 
Then $v_0:=(-(n-1),1,\cdots, 1)\in W$ is $H$-fixed, 
and by our method we obtain the family of categorical distributions on $\{1,\cdots, n\}$, 
\begin{align}
\{ s_1^{\delta_{1,x}}s_2^{\delta_{2,x}}\cdots s_n^{\delta_{n,x}}dc(x) \}_{(s_1,\cdots,s_n)\in\{s_i>0, \sum_{i=1}^n s_i=1\}}\quad  (x\in \{1,2,\cdots, n\}), \label{fami:cate}
\end{align}
where dc denotes the counting measure on $\{1,\cdots, n\}$ and $\delta_{i,x}$ is the Kronecker delta. 
\end{proposition}

\begin{proof}
Since $H=\frak{S}_{n-1}=(\frak{S}_n)_1$ is the stabilizer of the first element, $v_0=(-(n-1),1,\cdots, 1)$ is $H$-fixed. 
We have $W_0(G,H)=\{0\}$ from Fact~\ref{prop:triv_omega} (ii). 
By taking an inner product on $W$ which is the restriction of the standard inner product on $\R^n$, we identify $W^\vee$ with $W$. 
Then we have 
\begin{align*}
d\tilde{p}_{a_1,\cdots, a_n}(k)&=d\tilde{p}_{a_1,\cdots, a_n}((1,k)H)\\
&=e^{-(a_1,\cdots, a_n)\rho((1,k))\trans (-(n-1),1,\cdots,1)}dc(k)\\
&=e^{na_k}dc(k)\quad ((1,k)\in \frak{S}_n, k=1,\cdots, n). 
\end{align*}
Since this is integrable for any $(a_1,\cdots, a_n)\in W$, we have 
$\Theta=\{ (a_1,\cdots, a_n)\in W \}$ and 
$c_\theta=\frac{1}{\sum_{i=1}^n e^{na_i}}$. 
By a change of variables $s_k=\frac{e^{na_k}}{\sum_{i=1}^n e^{na_i}}$, 
we obtain the family of categorical distributions (\ref{fami:cate}). 
\end{proof}

\subsection{Gamma and inverse gamma distributions}
In this subsection, we realize the families of gamma and inverse gamma distributions on $\R_{>0}$ by using one dimensional representations of $\R_{>0}$. 
\begin{proposition}\label{prop:gamma}
Put $G=\R_{>0}$ and $H=\{1\}$. 
For $\lambda\in \R$, we denote by 
$\rho_\lambda\colon G\to \R^\times $ a representation given by
\[ \rho_\lambda(x)=x^\lambda.  \]
Then $v_0:=1\in \R$ is $H$-fixed, and by our method we obtain 
the following family of distributions on $\R_{>0}$, 
\begin{align}
\left\{ \frac{|\lambda|}{\Gamma(k)}\left(\frac{x}{\theta}\right)^{k\lambda}e^{-(\frac{x}{\theta})^\lambda}\frac{dx}{x}\right\}_{(k,\theta)\in \R_{>0}\times \R_{>0}}. \label{fami:param_gamma}
\end{align}
Especially, we obtain the family of gamma distributions if $\lambda=1$ and the family of inverse gamma distributions if $\lambda=-1$ as follows:
\begin{align}
&\left\{\frac{1}{\Gamma(k)}\left(\frac{x}{\theta}\right)^k e^{-\frac{x}{\theta}}\frac{dx}{x} \right\}_{(k,\theta)\in \R_{>0}\times \R_{>0}} \quad (\lambda=1),\label{fami:gamma}\\
&\left\{\frac{1}{\Gamma(k)}\left(\frac{\theta}{x}\right)^{k} e^{-\frac{\theta}{x}}\frac{dx}{x} \right\}_{(k,\theta)\in \R_{>0}\times \R_{>0}} \quad (\lambda=-1). \label{fami:inverse_gamma}
\end{align}
\end{proposition}

\begin{remark}
\begin{itemize}
\item By putting $k=1$ in the case of $\lambda=1$, we obtain the family of exponential distributions. 
\item By putting $k=\frac{n}{2}$ $(n\in \Z_{>0})$ and $\theta=2$ in the case of $\lambda =1$, 
we obtain the family of Chi-squared distributions. 
\item In the case of $\lambda =2$ and $k=1$, we obtain the family of Rayleigh distributions. 
\item In the case of $k=1$, we obtain the family of Weibull distributions with a shape parameter $\lambda>0$. 
\item In the case of $\lambda=-1$ and $k=\frac{1}{2}$, we obtain the family of unshifted L\'evy distributions. 
\end{itemize}
\end{remark}

\begin{proof}[Proof of Proposition~\ref{prop:gamma}]
From the continuity of $\tau \in W_0(G,H)$, we get 
$W_0(G,H)= \{\tau\colon x\mapsto \alpha\log x\ |\ \alpha\in \R\}$. 
We identify $\R^\vee$ with $\R$ by taking the standard inner product on $\R$. 
Then we have 
\[ d\tilde{p}_{\alpha, \beta}(x)=\exp(-\beta x^\lambda)x^\alpha \frac{dx}{x}. \]
Here, $\frac{dx}{x}$ is the Haar measure on $\R_{>0}$. 
\begin{claim}
We have 
\begin{align*}
\Theta&:=\left\{(\alpha, \beta)\in \R\times \R\ \Big|\ \int_{\R_{>0}}d\tilde{p}_{\alpha,\beta}<\infty \right\}\\
&=\{ (\alpha,\beta)\in \R\times \R\ |\ \beta>0, \alpha\lambda>0\},  \\
c_{\alpha,\beta}^{-1}&:=\int_{\R_{>0}}d\tilde{p}_{\alpha,\beta}=\frac{\Gamma(\frac{\alpha}{\lambda})}{|\lambda|\beta^{\frac{\alpha}{\lambda}}}. 
\end{align*}
\end{claim}
This claim can be proved by a direct calculation, so we omit the proof. 
As a result, we get
\[ dp_{\alpha,\beta}(x)=\frac{|\lambda|\beta^{\frac{\alpha}{\lambda}}}{\Gamma(\frac{\alpha}{\lambda})}x^\alpha e^{-\beta x^\lambda}\frac{dx}{x}. \]
By a change of variables $k=\frac{\alpha}{\lambda}$, $\theta=\beta^{-\frac{1}{\lambda}}$, 
we obtain the family (\ref{fami:param_gamma}). 
\end{proof}

\subsection{Wishart distributions}
In this subsection, we realize the family of Wishart distributions on the set $\Sym^+(n,\R)$ of all positive definite symmetric matrices of size $n$ by using a representation of $GL(n,\R)$.  
The case $n=1$ is essentially the same as the previous subsection. 

\begin{proposition}\label{prop:Wishart}
Let $n$ be a positive integer. 
Put $G=GL(n,\R)$ and $H=O(n)$. 
Let $\rho\colon G\to GL(\Sym(n,\R))$ be a representation as follows:
\[ \rho(g)S=gS\trans g \quad (S\in \Sym(n,\R)). \]
Then $v_0:=I_n$ is $H$-fixed, and by our method we obtain the family of Wishart distributions on $X=G/H\simeq \Sym^+(n,\R)$, 
\begin{align}\left\{\frac{(\det y)^\alpha}{\pi^{\frac{n(n-1)}{4}}\prod_{k=1}^n \Gamma(\alpha-\frac{k-1}{2})}\exp(-\Tr(yx))(\det x)^{\alpha-\frac{n+1}{2}}dx \right\}_{\{(y,\alpha)\in \Sym^+(n,\R)\times \R\ |\ \alpha>\frac{n-1}{2}\}}.\label{fami:Wishart} \end{align}
\end{proposition}

\begin{proof}[Proof of Proposition~\ref{prop:Wishart}]

\begin{claim}
$W_0(G,H)=\{ g\mapsto t\log |\det g|\ |\ t\in \R\}$. 
\end{claim}
This claim follows from $G/[G,G]\simeq \R^\times$, Remark~\ref{rem:factor} and Fact~\ref{fact:conti_hom_R}. 
We take an $G$-invariant measure $\frac{dx}{(\det x)^{\frac{n+1}{2}}}$ on $\Sym^+(n,\R)$ 
and identify $\Sym(n,\R)^\vee$ with $\Sym(n,\R)$ by the following inner product: 
\[\pairing{x}{y}=\Tr(xy)\quad (x,y\in \Sym(n,\R)).  \]
Then we have 
\[ d\tilde{p}_{y,\alpha}(x)= \exp(-\Tr(yx))(\det x)^\alpha \frac{dx}{(\det x)^\frac{n+1}{2}}. \]
Here we identify $GL(n,\R)/O(n)$ with $\Sym^+(n,\R)$ by $g O(n)\mapsto x=gI_n\trans g$. 
To normalize the above measure, we use Fact~\ref{normalizing_constant_Wishart}. 
As a result, we obtain Wishart distribution (\ref{fami:Wishart}). 
\end{proof}

\begin{fact}\label{fact:conti_hom_R}
$\{f\colon \R^\times \to \R \ |\ f \text{ is a continuous group homomorphism }\}=\{ x\mapsto \lambda\log |x|\ |\ \lambda\in \R\}$. 
\end{fact}

\begin{fact}[\cite{wishart,ingham,siegel}, see also \cite{ishi14} for example]\label{normalizing_constant_Wishart}
$\Theta=\{ (y,\alpha)\in \Sym(n,\R)\times \R\ |\ y\in \Sym^+(n,\R), \alpha>\frac{n-1}{2}\}$. 
\[ \int_{x\in \Sym^+(n,\R)} \exp(-\Tr(yx))(\det x)^\alpha \frac{dx}{(\det x)^\frac{n+1}{2}}=\pi^{\frac{n(n-1)}{4}}\prod_{k=1}^n \Gamma(\alpha-\tfrac{k-1}{2})(\det y)^{-\alpha}. \]
\end{fact}
\subsection{Von Mises distribution}
In this subsection, we realize the family of von Mises distributions on the circle $S^1$ by using the natural representation of $SO(2)$. 
\begin{proposition}\label{prop:Mises}
Put $G=SO(2)$ and $H=\{ I_2 \}$. 
Let $\rho\colon G\to GL(\R^2)$ be the natural representation. 
Then $v_0:=e_1\in \R^2$ is $H$-fixed, and by our method we obtain the family of von Mises distributions on $S^1$, 
\begin{align}\left\{ \frac{1}{2\pi I_0(\kappa)}e^{-\kappa \cos(x-\mu)}dx \right\}_{(\kappa, \mu) \in \R_{>0}\times \R/2\pi\Z}.  \label{fami:Mises} \end{align}
\end{proposition}
\begin{proof}
We use the identification $G/H=SO(2)/\{ I_2\}\simeq S^1\simeq \R/2\pi\Z$. 
From Fact~\ref{prop:triv_omega} (ii), 
we have $W_0(G,H)=\{0\}$. 
We identify $(\R^2)^\vee$ with $\R^2$ by taking the standard inner product on $\R^2$. 
Then we have 
\[ d\tilde{p}_{a,b}(x) = e^{-(a\cos x +b\sin x)}dx=e^{-\sqrt{a^2+b^2}\cos(x-\alpha)}dx, \]
where $dx$ is the uniform measure on $S^1$ with $\int_{S^1}dx=2\pi$, which is the Haar measure on $SO(2)$, and $\alpha\in \R/2\pi\Z$ satisfies $\cos \alpha=\frac{a}{\sqrt{a^2+b^2}}$, $\sin \alpha=\frac{b}{\sqrt{a^2+b^2}}$. 
Since $S^1$ is compact, we have $\Theta=\{\theta=(a,b)\in \R^2\}$, and 
$c_{\theta}^{-1}=2\pi I_0(\sqrt{a^2+b^2})$. 
Here $I_m(r)$ is the modified Bessel function of the first kind of order $m$ given as follows:
\begin{align} I_m(r)=\frac{1}{2\pi}\int_0^{2\pi} \cos (mx) e^{r\cos x} dx\quad (m\in \Z, r\in \R).  \label{modified_Bessel}\end{align}
Therefore by a change of variables $\kappa=\sqrt{a^2+b^2}$, $\mu=\alpha$, 
we obtain the family of von Mises distributions (\ref{fami:Mises}). 
\end{proof}

\subsection{Von Mises--Fisher distribution}
In this subsection, we realize the family of von Mises--Fisher distributions on the sphere $S^{n-1}$ by using the natural representation of $SO(n)$ for a positive integer $n\geq 2$. 

\begin{proposition}
Put $G=SO(n)$ and $H=SO(n-1)$. 
Let $\rho\colon G\to GL(\R^n)$ be the natural representation. 
Then $v_0:=e_1\in \R^n$ is $H$-fixed, and by our method 
we obtain the family of von Mises--Fisher distributions on $S^{n-1}$, 
\begin{align}
\left\{\frac{1}{(2\pi)^{\frac{n}{2}}\|\mu\|^{1-\frac{n}{2}}I_{\frac{n}{2}-1}(\|\mu\|)}\exp(-\trans \mu x)dx \right\}_{\mu \in \R^n}. 
\end{align}
Here, $dx$ is the uniform measure on $S^{n-1}$ with $\int_{S^{n-1}}dx=\frac{2\pi^{\frac{n}{2}}}{\Gamma(\frac{n}{2})}$. 
\end{proposition}
\begin{proof}
We use the identification $SO(n)/SO(n-1)\simeq S^{n-1}$, $gSO(n-1)\mapsto ge_1$. 
Since $G$ is compact, from Fact~\ref{prop:triv_omega} (ii), 
we obtain $W_0(G,H)=\{0\}$. 
We identify $(\R^n)^\vee$ with $\R^n$ by taking the standard inner product on $\R^n$. 
Then we have 
\begin{align*}
d\tilde{p}_\mu(x)=\exp(-\trans \mu x)dx, 
\end{align*}
where $dx$ is the uniform measure on $S^{n-1}$, which is $G$-invariant. 
Since $S^{n-1}$ is compact, we have $\Theta=\{\theta=\mu \in \R^n\}$. 
Therefore we obtain the family of von Mises--Fisher distributions on $S^{n-1}$ from Fact~\ref{fact:normalizing_const_for_MF}. 
\end{proof}
\begin{fact}[{\cite[Example~8.3]{Barndorff-Nielsen89}}]\label{fact:normalizing_const_for_MF}
For $\mu \in \R^n$, 
\[ \int_{S^{n-1}}\exp(-\trans \mu x)dx=(2\pi)^{\frac{n}{2}}\|\mu\|^{1-\frac{n}{2}}I_{\frac{n}{2}-1}(\|\mu\|). \]
Here $dx$ is the $SO(n)$-invariant measure on $S^{n-1}$ with $\int_{S^{n-1}}dx=\frac{2\pi^{\frac{n}{2}}}{\Gamma(\frac{n}{2})}$ and $I_m(r)$ denotes the modified Bessel function of the first kind of order $m$ (see (\ref{modified_Bessel})). 
\end{fact}

\subsection{Fisher--Bingham distribution}
In this subsection, we realize the family of Fisher--Bingham distributions on the sphere $S^{n-1}$ by using a representation of $SO(n)$ for a positive integer $n\geq 2$. 
\begin{proposition}
Put $G=SO(n)$ and $H=SO(n-1)=\left\{\begin{pmatrix}1 \\ & k \end{pmatrix}\ \Big|\ k \in SO(n-1)\right\}\subset G$. 
Let $\rho\colon G\to GL(\R^n\oplus \Sym(n,\R))$ be a representation given as follows:
\[ \rho(g)(v,S):=(gv , gS\trans g) \quad ((v, S)\in \R^n\oplus \Sym(n,\R)). \]
Then $v_0:=(e_1, E_{11})\in \R^n\oplus \Sym(n,\R)$ is $H$-fixed, and 
by our method we obtain the family of Fisher--Bingham distributions on $S^{n-1}$, 
\begin{align} \{c_{\mu, A}\exp(-\trans \mu x-\trans x A x)dx \}_{(\mu,A)\in \R^n\oplus \Sym(n,\R)}, \label{fami:Fisher--Bingham}\end{align}
where $c_{\mu, A}$ is the corresponding normalizing constant. 
\end{proposition}

\begin{remark}
We obtain the family of Kent distributions on $S^{n-1}$ as a subfamily of (\ref{fami:Fisher--Bingham}) if we assume the constraint $A\mu=0$. 
\end{remark}

\begin{proof}
We use the identification $SO(n)/SO(n-1)\simeq S^{n-1}$, $gSO(n-1)\mapsto ge_1$. 
Since $G=SO(n)$ is compact, from Fact~\ref{prop:triv_omega} (ii), 
we obtain $W_0(G,H)=\{0\}$. 
We identify $(\R^n\oplus \Sym(n,\R))^\vee$ with $\R^n\oplus \Sym(n,\R)$ by taking the direct sum of the standard inner product on $\R^n$ and the following inner product on $\Sym(n,\R)$: 
\[ \pairing{x}{y}=\Tr(xy)\quad (x,y\in \Sym(n,\R)). \]
Then we have for $x=g e_1\in S^{n-1}$ ($g\in SO(n)$), 
\begin{align*}
d\tilde{p}_{(\mu, A)}(x)&=\exp(-\pairing{(\mu, A)}{\rho(g)(e_1, E_{11})})dx\\
&=\exp(-\trans \mu x-\Tr(A x \trans x))dx\\
&=\exp(-\trans \mu x-\trans x A x)dx. 
\end{align*}
Since $S^{n-1}=G/H$ is compact, 
we have $\Theta=\R^n \oplus \Sym(n,\R)$. 
Therefore we obtain the family of Fisher--Bingham distributions (\ref{fami:Fisher--Bingham}). 
\end{proof}

\subsection{Hyperboloid distribution}
In this subsection, we realize the family of hyperboloid distributions \cite{jensen} on the $n$-dimensional hyperbolic space $H^n$ (see Definition~\ref{def:hyperbolic_space}) by using the natural representation of $SO_0(1,n)$. 

\begin{proposition}\label{prop:hyperboloid_dist}
Let $n$ be a positive integer and 
put $G=SO_0(1,n)$ and $H=SO(n)$. 
Let $\rho\colon G\to GL(\R^{n+1})$ be the natural representation. 
We denote by $e_0,\cdots,e_n$ the standard basis of $\R^{n+1}$. 
Then $v_0:=e_0\in \R^{n+1}$ is $H$-fixed, and by our method we obtain the family of hyperboloid distributions on the $n$-dimensional hyperbolic space $H^n$, 
\begin{align}
\{ c_\kappa\exp(\kappa\pairing{\xi}{v})d\nu(v)\}_{(\kappa,\xi)\in \R_{>0}\times H^n} \quad (v\in H^n),\label{eq:hyperboloid_dist} 
\end{align}
where $\pairing{\cdot}{\cdot}$ denotes a pseudo inner product given as (\ref{eq:1n_product}) below, $\nu$ is a $G$-invariant measure on $H^n$, which is described as 
$\sinh^{n-1}u\sin^{n-2} v_1\cdots \sin v_{n-2} du dv_1\cdots dv_{n-1}$ in hyperbolic-spherical coordinates $(u, v_1,\dots v_{n-1})$ (see section~2C in \cite{jensen} for details), 
and $c_\kappa$ is the normalizing constant written as follows:
\begin{align}
c_\kappa=\frac{\kappa^{\frac{n-1}{2}}}{(2\pi)^{\frac{n-1}{2}}2K_{\frac{n-1}{2}}(\kappa)}. \label{eq:normalizing_const_hyperbolid}
\end{align}
Here $K_\lambda$ is the modified Bessel function of the second kind and with index $\lambda$. 
\end{proposition}

\begin{remark}[see Section~7 in \cite{Barndorff-Nielsen78}, see also \cite{jensen}]
In \cite{jensen}, he denotes by $H^n$ the $(n-1)$-dimensional hyperbolic space. 
\begin{itemize}
\item In the case $n=1$, we have $c_\kappa=\frac{1}{2K_0(\kappa)}$. 
\item In the case $n=2$, we have $c_\kappa=\frac{\kappa e^\kappa}{2\pi}$. 
\end{itemize}
\end{remark}

\begin{definition}[hyperbolic space]\label{def:hyperbolic_space}
Let $n$ be a positive integer. 
The following manifold $H^n$ with the metric which is the restriction to $H^n$ of $-dx_0^2+dx_1^2+\cdots+dx_n^2$ on $\R^{n+1}$ is called $n$-dimensional hyperbolic space. 
\[ H^n=\{ x\in \R^{n+1}\ |\ x_0>0, \pairing{x}{x}=-1 \}.  \]
Here $\pairing{\cdot}{\cdot}$ is a pseudo inner product on $\R^{n+1}$ given as follows:
\begin{align} 
\pairing{x}{y}:=-x_0y_0+x_1y_1+\cdots +x_{n}y_{n}. \label{eq:1n_product}
\end{align}
\end{definition}

A Lie group $SO_0(1,n)$, which is the identity component of $SO(1,n)$, 
acts transitively on the $n$-dimensional hyperbolic space $H^n$. 
Since the stabilizer of $e_1\in H^n$ is $SO(n)$, 
we identify $H^n$ with the homogeneous space $SO_0(1,n)/SO(n)$ as follows:
\begin{align}
SO_0(1,n)/SO(n)&\to H^n \label{eq:homo_to_hyper},\\
g SO(n) &\mapsto g\cdot e_0. 
\end{align}
\begin{proof}[Proof of Proposition~\ref{prop:hyperboloid_dist}]
We identify $(\R^{n+1})^\vee$ with $\R^{n+1}$ by the pseudo inner product (\ref{eq:1n_product}). 
We have $W_0(G,H)=\{0\}$ from Fact~\ref{prop:triv_omega} (iii). 
Take the $G$-invariant measure $\nu$ on $H^n$ as in Proposition~\ref{prop:hyperboloid_dist}. 
Then we have 
\[ d\tilde{p}_y(v)=\exp(-\pairing{y}{v})d\nu(v)\quad (y\in \R^{n+1}, v=ge_0\in H^n). \]

\begin{fact}[\cite{Barndorff-Nielsen78,jensen}]
$\Theta=\{y\in \R^{n+1}\ |\ y_0<0, \pairing{y}{y}<0 \}$. 
\end{fact}
Since there is a natural bijection $\R_{>0}\times H^n \to \Theta$, 
$(\kappa,\xi)\mapsto -\kappa\xi$, we use $\R_{>0}\times H^n$ as a parameter space. 
\begin{fact}[Section~7 in \cite{Barndorff-Nielsen78}]
The normalizing constant of $\exp(-\pairing{y}{v})d\nu(v)=\exp(\kappa \pairing{\xi}{v})d\nu(v)$ is given as (\ref{eq:normalizing_const_hyperbolid}). 
\end{fact}
Therefore, we obtain the family of hyperboloid distributions (\ref{eq:hyperboloid_dist}). 
\end{proof}

\subsubsection{A family of distributions on the upper half plane}\label{sec:poincare_dist}
In this subsubsection, via a diffeomorphism between $H^2$ and the upper half plane $\mathcal{H}$, we realize the corresponding family of distributions on the upper half plane to the family above of hyperboloid distributions for $n=2$. 
%
Since the exponential family on the upper half plane is compatible with the Poincar\'e metric, in this paper, 
we call it the family of {\it Poincar\'e distributions}, which is given as follows: 
\begin{align}\left\{ \frac{D\ e^{2D}}{\pi} \exp\left(- \frac{a(x^2+y^2)+2bx+c }{y} \right)\frac{dxdy}{y^2} \right\}_{\begin{pmatrix}a & b\\b &c\end{pmatrix}\in \Sym^+(2,\R)}, \label{def:poincare_dist}\end{align}
where $D=\sqrt{ac-b^2}$.

First, let us recall the definition of the upper half plane. 
\begin{definition}[upper half plane]
The following manifold $\mathcal{H}$ with Poincar\'e metric $\frac{dx^2+dy^2}{y^2}$ is called upper half plane:
\[ \mathcal{H}=\{ x+iy \in \C\ |\ y>0\}. \]
\end{definition}
The Lie group $SL(2,\R)$ acts transitively on the upper half plane $\mathcal{H}$ by the M\"obius transformations given by:
\[ g\cdot z:= \frac{az+b}{cz+d} \quad \left(g=\begin{pmatrix}a & b \\ c& d\end{pmatrix}\in SL(2,\R)\right). \]
Since the stabilizer of $i$ is $SO(2)$, we identify $\mathcal{H}$ with 
the homogeneous space $SL(2,\R)/SO(2)$ as follows:
\begin{align}  
\mathcal{H}&\leftrightarrow SL(2,\R)/SO(2)\label{eq:upper_half_homo},\\
z=x+iy&\mapsto\begin{pmatrix}\sqrt{y} & \frac{x}{\sqrt{y}}\\0 & \frac{1}{\sqrt{y}} \end{pmatrix}SO(2), \\
g \cdot i& \mapsfrom g SO(2). 
\end{align}

Moreover, we use the following diffeomorphism:
\begin{align}
SL(2,\R)/SO(2)\to SO_0(1,2)/SO(2), \label{eq:diffeo_bet_homo}
\end{align}
which is induced by the following $2:1$ group homomorphism:
\begin{align}
SL(2,\R)\to SO_0(1,2), g\mapsto \Ad(g). 
\end{align}
Here since $SL(2,\R)$ is connected and the adjoint representation $\Ad\colon SL(2,\R)\to GL(\frak{sl}(2,\R))$ 
preserves the pseudo inner product $\frac{1}{2}\Tr(xy)$ with signature $(2,1)$ on $\frak{sl}(2,\R)$, the above map is well-defined. 
We use a basis $f_1=\begin{pmatrix}0 & -1 \\1 &0\end{pmatrix}$, $f_2=\begin{pmatrix}0 & 1\\1 & 0\end{pmatrix}$, $f_3=\begin{pmatrix}1 & 0 \\ 0 & -1\end{pmatrix}$ of $\frak{sl}(2,\R)$. 
Since $SO(2)\subset SL(2,\R)$ stabilizes $f_1$, we have $\Ad(SO(2))=\left\{\begin{pmatrix}1 & \\ & g\end{pmatrix}\in SO_0(1,2)\ \Big|\ g\in SO(2)\right\}$. 
From the composition of (\ref{eq:upper_half_homo}), (\ref{eq:diffeo_bet_homo}) and (\ref{eq:homo_to_hyper}), we have the following correspondence:
\begin{align}
\mathcal{H}&\to H^2,\\
x+iy&\mapsto \frac{1}{2y}\begin{pmatrix}1+x^2+y^2\\1-(x^2+y^2)\\ 2x \end{pmatrix}=\begin{pmatrix}v_0\\v_1\\v_2\end{pmatrix}. 
\end{align}
Therefore, by putting $a=\frac{\xi_0+\xi_1}{2\pi}\kappa$, $b=\frac{\xi_2}{2}\kappa$ and $c=\frac{\xi_0-\xi_1}{2}\kappa$, we obtain the Poincar\'e distribution (\ref{def:poincare_dist}) from the hyperboloid distribution 
$\{\frac{\kappa e^\kappa}{2\pi}\exp(\kappa(-\xi_0v_0+\xi_1v_1+\xi_2v_2))d\nu(v)\}_{(\kappa,\xi)\in \R_{>0}\times H^2}$.

\section*{Acknowledgement}
The authors would like to thank Professors Hideyuki Ishi, Fr\'ed\'eric Barbaresco, Toshiyuki Kobayashi, Kenichi Bannai, Kei Kobayashi and Shintaro Hashimoto for valuable comments.  
The authors would also like to thank Atsushi Suzuki, Kenta Oono and Kentaro Minami for helpful comments and discussion. 
This work was supported by JST, ACT-X Grant Number JPMJAX190K, Japan. 




\begin{thebibliography}{99}
\bibitem{Barndorff-Nielsen69} {\sc O. E.~Barndorff-Nielsen}, {\it L\'evy homeomorphic parametrization and exponential families}, Z. Wahrscheinlichkeitstheorie verw. Geb. 12 (1969), 56--58. 
\bibitem{Barndorff-Nielsen70} {\sc O. E.~Barndorff-Nielsen}, {\it Exponential families-exact theory}, Various Publication Series No. 19. Aarhus Universitet, Aarhus, (1970). 
\bibitem{Barndorff-Nielsen78a} {\sc O. E.~Barndorff-Nielsen}, {\it Information and exponential families in statistical theory} Chichester: Wiley, (1978). 
\bibitem{Barndorff-Nielsen78} {\sc O. E.~Barndorff-Nielsen}, {\it Hyperbolic distributions and distribution on hyperbolae}, Scand. J. Statist. {\bf 8} (1978), 151--157. 
\bibitem{Barndorff-Nielsen82} {\sc O. E.~Barndorff-Nielsen, P.~Bla\ae sild, J. L.~Jensen, B.~J\o rgensen}, {\it Exponential transformation models}, Proc. R. Soc. Lond. A~{\bf 379} (1982), 41--65. 
\bibitem{Barndorff-Nielsen89} {\sc O. E.~Barndorff-Nielsen, P.~Bla\ae sild, P. S.~Eriksen}, {\it Decomposition and invariance of measures, and statistical transformation models}, New York: Springer Verlag, (1989). 
\bibitem{cohen-welling} {\sc T.~S.~Cohen, M.~Welling}, {\it Harmonic exponential families on manifolds}, In Proceedings of the 32nd International Conference on Machine Learning (ICML), volume~37 (2015), 1757--1765. 
\bibitem{diaconis-ylvisoker} {\sc P.~Diaconis, F.~Ylvisoker}, {\it Conjugate priors for exponential families}, The Annals of Statistics, Vol.7, No.2 (1979), 269--281. 
\bibitem{hornic-grun} {\sc K.~Hornik, B.~Gr\"un}, {\it On conjugate families and Jeffreys priors for von Mises--Fisher distributions}, Jour. of Stat. Plan. and Inf. {\bf 143} (2013), 992--999. 
\bibitem{ingham} {\sc A.~E.~Ingham}, {\it An integral which occurs in statistics}, Proc. Cambr. Philos. Soc. {\bf 29} (1933), 271--276. 
\bibitem{ishi14} {\sc H.~Ishi}, {\it Homogeneous cones and their application to statistics. Modern methods of multivariate statistics}, Lecture notes of the CIMPA-FECYT-UNESCO-ANR Workshop held in Hammamet, September 2011. (eds. by P.~Graczyk and A.~Hassairi),  Travaux en Cours, vol. 82, Hermann \'Editeurs, Paris, (2014). 
\bibitem{jensen} {\sc J.~L.~Jensen}, {\it On the hyperboloid distribution}, Scand. J. Statist. {\bf 8} (1981), 193--206. 
\bibitem{kume-walker} {\sc K.~Kume, S.~G.~Walker}, {\it On the Fisher--Bginham distribution}, Stat. Comput. {\bf 19} (2009), 167--172. 
\bibitem{wijsman} {\sc R.~A.~Wijsman}, {\it 7. Invariant and relatively invariant measures on locally compact groups and spaces. Invariant measures on groups and their use in statistics}, Inst. of Math. Stat., Hayward, CA, (1990), 119--150. 
\bibitem{siegel} {\sc G.~L.~Siegel} {\it \"Uer die analytische tehorie der quadratische Formen}, Ann. of Math. {\bf 36} (1935), 527--606. 
\bibitem{ty21} {K.~Tojo, T.~Yoshino}, {\it Harmonic exponential families on homogeneous spaces}, Info. Geo. {\bf 4} (2021), 215--243. 
\bibitem{wishart} {\sc J.~Wishart}, {\it The generalized product moment distribution in samples from a normal multivariate population}, Biometrika {\bf 20A} (1928), 32--52. 
\end{thebibliography}
\end{document}